\documentclass[12pt, leqno]{article}
\usepackage{amsmath,amssymb,latexsym,amsthm,mathrsfs}
\usepackage{cite}
\usepackage[paper=a4paper,left=30mm,right=20mm,top=25mm,bottom=30mm]{geometry}
\usepackage{amsmath}
\usepackage{amssymb}
\usepackage{amsthm}
\usepackage{amsthm,thmtools}
\usepackage{array} 
\usepackage{calc,ifthen}
\usepackage{graphicx} 
\usepackage{color,transparent}
\theoremstyle{plain} 
\newtheorem{defn}{Definition}[section]
\newtheorem*{defn*}{Definition}
\newtheorem*{lem*}{Lemma}
\newtheorem{lem}{Lemma}[section]
\newtheorem{thm}{Theorem}[section]
\newtheorem*{thm*}{Theorem}
\newtheorem{cor}{Corollary}[section]
\newtheorem*{cor*}{Corollary}
\newtheorem{pro}{Property}[section]
\newtheorem*{pro*}{Property}
\newtheorem{prop}{Proposition}[section]
\newtheorem*{prop*}{Proposition}
\newcommand{\xth}{$x^{\text{th }}$}
\newcommand\Lcm{\operatorname{lcm}}
\title{Generalizing the Abundancy of an Integer}
\author{David C. Luo \thanks{
                \texttt david.luo@emory.edu}
  \\ \\ Department of Mathematics, Emory University}
\begin{document} 
%%%%%%%%%%%%%%%%%%%%%%%%%%%%%%%%%%%%%%%%%%%%%%%%%%%%%%%%%
\date{}
\maketitle
%%%%%%%%%%%%%%%%%%%%%%%%%%%%%%%%%%%%%%%%%%%%%%%%%%%%%%%%%

%%%%%%%%%%%%%%%%%%%%%%%%%%%%%%%%%%%%%%%%%%%%%%%%%%%%%%%%%
\begin{abstract}
The abundancy index of a positive integer is the ratio between the sum of its divisors and itself. We generalize previous results on abundancy indices by defining a two-variable abundancy index function as $I(x,n)\colon\mathbb{Z^+}\times\mathbb{Z^+}\to\mathbb{Q}$ where $I(x,n)=\frac{\sigma_x(n)}{n^x}$. Specifically, we extend limiting properties of the abundancy index and construct sufficient conditions for rationals greater than one that fail to be in the image of the function $I(x,n)$. 
\end{abstract}
%%%%%%%%%%%%%%%%%%%%%%%%%%%%%%%%%%%%%%%%%%%%%%%%%%%%%%%%%

%%%%%%%%%%%%%%%%%%%%%%%%%%%%%%%%%%%%%%%%%%%
\section{Introduction}
%%%%%%%%%%%%%%%%%%%%%%%%%%%%%%%%%%%%%%%%%%%
The concept of \textit{perfect} numbers is one of the oldest mysteries in number theory and has been a major topic of study for over two millennia. Throughout the ages, perfect numbers have been perceived to possess superstitious properties \cite{Voight}. For example, the Pythagoreans related the perfect number six to marriage, health, and beauty \cite{Voight}. On the other hand, early Hebrews distinguished six as a ``truly" perfect number as they believed that God created the Earth in six days \cite{Voight}. Although perfect numbers are important in ancient belief systems and superstitions, they also play a prominent role in mathematical theory. As Nicomachus pointed out, perfect numbers create a balance between deficient (numbers whose proper divisors sum to less than the number itself) and abundant (numbers whose proper divisors sum to greater than the number itself) numbers \cite{Voight}. A noteworthy result proven by Euler characterizes even perfect numbers in a specific form \cite{Voight}. 
%%%%%%%%%%%%%%%%%%%%%%%%%%%%%%%%%%%%%%%%%%%
\begin{defn}
A positive integer $N$ is perfect if and only if $N$ is equal to the sum of its proper divisors. 
\end{defn}
%%%%%%%%%%%%%%%%%%%%%%%%%%%%%%%%%%%%%%%%%%%
\begin{thm}\label{a2}
(Euler) Even perfect numbers are of the form $N =2^{p-1}(2^p -1)$, where $p$ and $(2^p -1)$ are primes.
\end{thm}
%%%%%%%%%%%%%%%%%%%%%%%%%%%%%%%%%%%%%%%%%%%
Open problems regarding perfect numbers include the questionable existence of an odd perfect number and infinitely many perfect numbers. For further study of perfect numbers, the \textit{abundancy index} of a positive integer was introduced as the ratio between the sum of its divisors and itself. 
%%%%%%%%%%%%%%%%%%%%%%%%%%%%%%%%%%%%%%%%%%%
\begin{defn}
The \textit{abundancy index} of a positive integer $n$, $I(n)$, is defined by 
$$
I(n)\colon \mathbb{Z^+}\to\mathbb{Q}
$$
where
$$
 I(n)=\frac{\sigma(n)}{n}.
$$
\end{defn}
%%%%%%%%%%%%%%%%%%%%%%%%%%%%%%%%%%%%%%%%%%%
In particular, we have that a positive integer is perfect if and only if it has an abundancy index of two. By studying the abundancy index, we gain extended properties of perfect numbers. The following theorem lists criteria for finding an odd perfect number \cite{Holdener}. 
%%%%%%%%%%%%%%%%%%%%%%%%%%%%%%%%%%%%%%%%%%%
\begin{thm}\label{a0}
%%%%%%%%%%%%%%%%%%%%%%%%%%%%%%%%%%%%%%%%%%%
There exists an odd perfect number if and only if there exist positive integers $p, n,$ and $\alpha$ such that $p \equiv \alpha \equiv 1 \left(\text{mod  }4\right)$, where $p$ is a prime not dividing $n$, and
%%%%%%%%%%%%%%%%%%%%%%%%%%%%%%%%%%%%%%%%%%%
\begin{align*}
I(n) = \frac{2p^{\alpha}(p-1)}{p^{\alpha+1}-1}.
\end{align*}
%%%%%%%%%%%%%%%%%%%%%%%%%%%%%%%%%%%%%%%%%%%
\end{thm}
%%%%%%%%%%%%%%%%%%%%%%%%%%%%%%%%%%%%%%%%%%%
Theorem \ref{a0} asserts that if we can find a positive integer $n$ with an abundancy index of $\frac{13}{7}$ such that $13$ does not divide $n$, then we know that an odd perfect number exists. A question one might ask is whether or not \textit{some} positive integer meets these requirements. To answer this, we categorize rationals greater than one that fail to be the abundancy index of any positive integer. We call these rationals \textit{abundancy outlaws}. Much progress has been made in determining the status of rational numbers greater than one as abundancy outlaws or indices. One notable result generates a class of abundancy outlaws of the form $\frac{\sigma(n)-t}{n}$, where $t$ is a positive integer \cite{Holdener}.
%%%%%%%%%%%%%%%%%%%%%%%%%%%%%%%%%%%%%%%%%%%
\begin{thm} \label{1}
Let $m$ and $k$ be integers. If $\left(k, m\right)=1$ and $m<k<\sigma(m)$, then $\frac{k}{m}$ is an abundancy outlaw. 
\end{thm} 
%%%%%%%%%%%%%%%%%%%%%%%%%%%%%%%%%%%%%%%%%%%
In 2007, Judy Holdener and William Stanton proved that under certain conditions, rationals of the form $\frac{\sigma(n)+t}{n}$ are also abundancy outlaws, where $t$ is a positive integer \cite{Holdener}. This theorem proves to be extremely useful as it extends Theorem \ref{1} and classifies abundancy outlaws in a similar form.
%%%%%%%%%%%%%%%%%%%%%%%%%%%%%%%%%%%%%%%%%%%
\begin{thm} \label{2}
For a positive integer t, let $\frac{\sigma(N)+t}{N}$ be a fraction in lowest terms, and let $N= \prod_{i=1}^{n} p_i^{k_i}$ for primes $p_1, p_2,\ldots, p_n$. If there exists a positive integer $j \leq n$ such that $p_j<\frac{1}{t}\sigma\left(\frac{N}{p_j^{k_j}}\right)$ and $\sigma(p_j^{k_j})$ has a divisor $D > 1$ such that at least one of the following is true: 
\begin{enumerate}
\item $I\left(p_j^{k_j}\right)I(D) > \frac{\sigma(N)+t}{N}$ and $gcd(D,t)=1$; and 
\item $gcd(D, Nt) = 1$,
\end{enumerate}
then $\frac{\sigma(N)+t}{N}$ is an abundancy outlaw. 
\end{thm}
%%%%%%%%%%%%%%%%%%%%%%%%%%%%%%%%%%%%%%%%%%%
Additionally, Holdener and Stanton were also able to show that certain rationals $\frac{a}{b}$ greater than one falling within the range $I(n)<\frac{a}{b}<I(p_in)$ where $n$ is a positive integer and $p_i$ is a prime divisor of $n$ are abundancy outlaws \cite{Holdener}. 
%%%%%%%%%%%%%%%%%%%%%%%%%%%%%%%%%%%%%%%%%%%
\begin{thm} \label{12}
Let $\frac{r}{s}$ be a fraction in lowest terms such that there exists a divisor $N= \prod_{i=1}^{n} p_i^{k_i}$ of $s$ satisfying the following two conditions:
\begin{enumerate}
\item $\frac{r}{s} < I(p_iN)$ for all $i\leq n$
\item The product $\sigma(N)(\frac{s}{N})$ has a divisor $M$ such that $(M,r) =1$ and $I(M) \geq \frac{\sigma\left(p_j^{k_j+1}\right)}{\sigma\left(p_j^{k_j+1}\right)-1}$ for some positive integer $j\leq n$.
\end{enumerate}
Then $\frac{r}{s}$ is an abundancy outlaw. 
\end{thm}
%%%%%%%%%%%%%%%%%%%%%%%%%%%%%%%%%%%%%%%%%%%
In the summer of 2007, Judy Holdener and Laura Czarnecki received the following theorem and corollary dealing with abundancy indices \cite{Czarnecki}. In doing so, they were able to identify certain rationals that are the abundancy index of at least one positive integer.
%%%%%%%%%%%%%%%%%%%%%%%%%%%%%%%%%%%%%%%%%%%
\begin{thm}\label{3}
If $\frac{a}{b}$ is a fraction greater than one in reduced form, $\frac{a}{b}=I(N)$ for some $N\in\mathbb{N}$, and b has a divisor $D=\prod_{i=1}^{n} p_i^{k_i}$ such that $I(p_iD)>\frac{a}{b}$ for all $1 \leq i \leq n$, then $\frac{D}{\sigma(D)}\frac{a}{b}$ is an abundancy index as well.
\end{thm}
%%%%%%%%%%%%%%%%%%%%%%%%%%%%%%%%%%%%%%%%%%%
\begin{cor}
Let $m, n, t \in\mathbb{N}$. If $\frac{\sigma(mn)+\sigma(m)t}{mn}$ is in reduced form with $m= \prod_{i=1}^{l} p_i^{k_i}$ and $I(p_im)>\frac{\sigma(mn)+\sigma(m)t}{mn}$ for all $1\leq i \leq l$, then $\frac{\sigma(n)+t}{n}$  is an abundancy index if $\frac{\sigma(mn)+\sigma(m)t}{mn}$ is an abundancy index. 
\end{cor}
%%%%%%%%%%%%%%%%%%%%%%%%%%%%%%%%%%%%%%%%%%%
Our main goal is to generalize and extend previous properties of the abundancy index, specifically, results regarding abundancy outlaws and upper bounds. We begin by defining a two-variable abundancy index function as the \textit{\xth abundancy index} to consider the ratio between the sum of the divisors of a positive integer $n$ raised to a power $x$ and $n^x$. 
%%%%%%%%%%%%%%%%%%%%%%%%%%%%%%%%%%%%%%%%%%%
\begin{defn}
The \textit{sum-of-divisors} function of a positive integer $n$, $\sigma_x(n)$, is defined by 
$$
\sigma_x\left(n\right) = \sum_{d|n}d^x.
$$
\end{defn}
%%%%%%%%%%%%%%%%%%%%%%%%%%%%%%%%%%%%%%%%%%%
\begin{defn}
The \textit{\xth abundancy index} of a positive integer $n$, $I(x,n)$, is defined by 
$$
I(x,n)\colon \mathbb{Z^+}\times\mathbb{Z^+}\to\mathbb{Q}
$$
where
$$
 I(x,n)=\frac{\sigma_x(n)}{n^x}.
$$
\end{defn}
%%%%%%%%%%%%%%%%%%%%%%%%%%%%%%%%%%%%%%%%%%%%
We observe characteristics and identify which rationals greater than one lie in the image of the \xth abundancy index by genearlizing Holdener, Stanton, and Czarnecki's work. Similarly, we call rationals greater than one that fail to be in the image of the function $I(x,n)$ \textit{\xth abundancy outlaws}. The four theorems to follow generalize Theorems \ref{1}, \ref{2}, \ref{12}, and \ref{3} respectively. The proofs and greater explanations will be demonstrated in later sections. 
%%%%%%%%%%%%%%%%%%%%%%%%%%%%%%%%%%%%%%%%%%%
\begin{restatable}{thm}{a}\label{Theorem1}
Let $m$ and $k$ be positive integers. If $\left(k, m^x \right)=1$, and $m^x<k\leq \sigma_x(m)$, then $\frac{k}{m^x}$ is an \xth abundancy outlaw. 
\end{restatable} 
%%%%%%%%%%%%%%%%%%%%%%%%%%%%%%%%%%%%%%%%%%%
\begin{restatable}{thm}{b}\label{Theorem2}
For a positive integer $t$, let $\frac{\sigma_x(n)+t}{n^x}$ be a fraction such that $\left(\sigma_x(n)+t, n^x\right)=1$, and let $n^x= \prod_{i=1}^{s} p_i^{xk_i}$. Suppose that there exists a positive integer $1\leq j \leq s$ such that $p_j^x<\frac{1}{t}\sigma_x\left(\frac{n}{p_j^{k_j}}\right)$ and suppose further that $\sigma_x\left(p_j^{k_j}\right)$ has a divisor $d^x $ greater than one such that at least one of the following is true:
\begin{enumerate}
\item $I\left(x,p_j^{k_j}\right)I(x,d) > \frac{\sigma_x(n)+t}{n^x}$ and $(d^x,t)=1$; or
\item $(d^x, n^xt) = 1$.
\end{enumerate}
Then $\frac{\sigma_x(n)+t}{n^x}$ is an \xth abundancy outlaw. 
\end{restatable} 
%%%%%%%%%%%%%%%%%%%%%%%%%%%%%%%%%%%%%%%%%%%
\begin{restatable}{thm}{e}\label{Theoremd}
Let $\frac{k}{lm^x}$ be a fraction greater than one such that $(k,lm^x)=1$. If there exists a divisor $n^x= \prod_{i=1}^{s} p_i^{xk_i}$ of $lm^x$ such that  
\begin{enumerate}
\item $\frac{k}{lm^x}<I\left(x,p_in\right)$ for all $1\leq i\leq s$, and 
\item $\sigma_x(n)l\left(\frac{m}{n}\right)^x$  has a divisor $d^x$ such that $(d^x,k)=1$ and $I(x,d)\geq \frac{\sigma_x\left(p_j^{k_j+1}\right)}{\sigma_x\left(p_j^{k_j+1}\right)-1}$ for some positive integer $1\leq j\leq s$, 
\end{enumerate}
then $\frac{k}{lm^x}$ is an \xth abundancy outlaw. 
\end{restatable}
%%%%%%%%%%%%%%%%%%%%%%%%%%%%%%%%%%%%%%%%%%%
\begin{restatable}{thm}{c}
Suppose that $\frac{a}{cb^x}$ is a fraction greater than one in simplest terms, $\frac{a}{cb^x}=I(x,n)$ for some positive integer $n$, and $cb^x$ has a divisor $d^x=\prod_{i=1}^{s} p_i^{xk_i}$ such that $I(x,p_id)>\frac{a}{cb^x}$ for all $1 \leq i \leq s$. Then $\frac{d^x}{\sigma_x(d)}\frac{a}{cb^x}$ is an \xth abundancy index as well.
\end{restatable} 
%%%%%%%%%%%%%%%%%%%%%%%%%%%%%%%%%%%%%%%%%%%
In addition, we build off results we use to locate \xth abundancy outlaws and extend properties relating to limiting values and upper bounds of the abundancy index. Two well known properties bound the abundancy index in relation to prime powers \cite{Weiner}. 
%%%%%%%%%%%%%%%%%%%%%%%%%%%%%%%%%%%%%%%%%%%
\begin{pro}\label{p1}
%%%%%%%%%%%%%%%%%%%%%%%%%%%%%%%%%%%%%%%%%%%%%%
For any prime power $p^{r}$, the following inequality holds
$$
\frac{\sigma(p^{r})}{p^{r}}<\frac{p}{p-1}.
$$
%%%%%%%%%%%%%%%%%%%%%%%%%%%%%%%%%%%%%%%%%%%%%%
\end{pro}
%%%%%%%%%%%%%%%%%%%%%%%%%%%%%%%%%%%%%%%%%%%%%%
\begin{pro}\label{p2}
%%%%%%%%%%%%%%%%%%%%%%%%%%%%%%%%%%%%%%%%%%%%%%
For any integer $n>1$ and prime $p$ that divides $n$, 
$$
\frac{\sigma(n)}{n}<\prod_{p|n} \frac{p}{p-1}=\prod_{p|n}\left(1+\frac{1}{p-1}\right).
$$ 
%%%%%%%%%%%%%%%%%%%%%%%%%%%%%%%%%%%%%%%%%%%%%%
\end{pro}
%%%%%%%%%%%%%%%%%%%%%%%%%%%%%%%%%%%%%%%%%%%%%%
The examination we consider categorizes positive integers of the form $nm^k$, where $n$,$m$ are positive integers and $k$ is a nonnegative integer. By applying this categorization, we can find $\displaystyle\lim_{k \to \infty}I\left(x,nm^k\right)$ for any $n$ and $m$. This enables us to know the limiting value for any combination of positive integers, rather than prime powers alone. The main result we obtain is listed the following proposition.
%%%%%%%%%%%%%%%%%%%%%%%%%%%%%%%%%%%%%%%%%%%%%%
\begin{restatable}{prop}{d}\label{PropositionPerfect}
Let $n$ and $m$ be positive integers and $k$ a nonnegative integer with $m$ having the prime factorization $m=p_1^{s_1}p_2^{s_2}\cdots p_t^{s_t}$. If $n=ab$, where $a$ is the largest divisor of $n$ such that $\left(a,m\right)=1$, then
\begin{equation*}
\displaystyle{\lim_{k \to \infty}}I\left(x, nm^k\right) = I\left(x,a\right) \prod_{i=1}^{t} \frac{p_i^x}{p_i^x-1}.
\end{equation*}
\end{restatable}
%%%%%%%%%%%%%%%%%%%%%%%%%%%%%%%%%%%%%%%%%%%%

%%%%%%%%%%%%%%%%%%%%%%%%%%%%%%%%%%%%%%%%%%%%
\section{Preliminaries}
%%%%%%%%%%%%%%%%%%%%%%%%%%%%%%%%%%%%%%%%%%%%
In this section, we present additional definitions and notations we use. From our previous introduction of the abundancy index, we attain the idea of abundancy outlaws, rationals greater than one that fail to be in the image of the function $I(n)$.
\begin{defn}
A rational number greater than one is an abundancy outlaw if fails to be the abundancy index of any positive integer. 
\end{defn}
%%%%%%%%%%%%%%%%%%%%%%%%%%%%%%%%%%%%%%%%%%%%
We generalize this concept to the \xth abundancy index by introducing the notion of an \textit{\xth abundancy outlaw}. 
%%%%%%%%%%%%%%%%%%%%%%%%%%%%%%%%%%%%%%%%%%%%
\begin{defn}
A rational number greater than one is an \xth abundancy outlaw if it fails to be the \xth abundancy index of any positive integer. 
\end{defn}
%%%%%%%%%%%%%%%%%%%%%%%%%%%%%%%%%%%%%%%%%%%%
Note that in this paper, we refer to the abundancy index and abundancy outlaw as the \xth abundancy index and \xth abundancy outlaw respectively, when $x=1$. Next we take a look at multiplicative properties of the \xth abundancy index. Let $\left(a,b\right)$ denote the greatest common divisor of $a$ and $b$. Since $\sigma_x$ is multiplicative, $I(x,n)$ is also multiplicative; that is, if $\left(a,b\right)=1$, then by \cite{Voight},
%%%%%%%%%%%%%%%%%%%%%%%%%%%%%%%%%%%%%%%%%%%%
$$ 
I\left(x,ab\right) = I\left(x,a\right)I\left(x,b\right).
$$
%%%%%%%%%%%%%%%%%%%%%%%%%%%%%%%%%%%%%%%%%%%%
It is known that for any positive integers $a$ and $b$, $ab =\left(a,b\right)\cdot \Lcm\left(a,b\right)$ where $\Lcm\left(a,b\right)$ denotes the least common multiple of $a$ and $b$ \cite{Ferguson}. We apply this property to the \xth abundancy index. 
%%%%%%%%%%%%%%%%%%%%%%%%%%%%%%%%%%%%%%%%%%%%
\begin{prop}
For any positive integers $a$ and $b$,
\[
I(x,a)I(x,b) = I\left(x,(a,b)\right)I\left(x,\Lcm (a,b)\right).
\]
\end{prop}
%%%%%%%%%%%%%%%%%%%%%%%%%%%%%%%%%%%%%%%%%%%%%
\begin{proof}
Let $a$ and $b$ be positive integers having the following prime factorizations
%%%%%%%%%%%%%%%%%%%%%%%%%%%%%%%%%%%%%%%%%%%%%
\begin{align*}
a & = p_{1}^{r_1} p_{2}^{r_2} \cdots p_{m}^{r_m} \\ 
b & = p_{1}^{s_1} p_{2}^{s_2} \cdots p_{m}^{s_m}
\end{align*}
%%%%%%%%%%%%%%%%%%%%%%%%%%%%%%%%%%%%%%%%%%%%%
where $r_i$ and $s_i$ are nonnegative integers for all $1\leq i \leq m$. Since $I(x,n)$ is multiplicative, we have that 
%%%%%%%%%%%%%%%%%%%%%%%%%%%%%%%%%%%%%%%%%%%%%
\begin{align*}
I\left(x,a\right)I\left(x,b\right) 
& = I\left(x,p_{1}^{r_1}p_{2}^{r_2}\cdots p_{m}^{r_m} \right) I\left(x,p_{1}^{s_1}p_{2}^{s_2} \cdots p_{m}^{s_m}\right) \\
%%%%%%%%%%%%%%%%%%%%%%%%%%%%%%%%%%%%%%%%%%%%%
& =  I\left(x,p_{1}^{r_1}\right)I\left(x,p_{2}^{r_2}\right) \cdots I\left(x,p_{m}^{r_m}\right)I\left(x,p_{1}^{s_1}\right)I\left(x,p_{2}^{s_2} \right) \cdots I\left(x,p_{m}^{s_m}\right) \\ 
%%%%%%%%%%%%%%%%%%%%%%%%%%%%%%%%%%%%%%%%%%%%%
& = I\left(x,p_{1}^{r_1}\right)I\left(x,p_{1}^{s_1}\right) I\left(x,p_{2}^{r_2}\right)I\left(x,p_{2}^{s_2}\right) \cdots I\left(x,p_{m}^{r_m}\right)I\left(x,p_{m}^{s_m}\right).
\end{align*}
%%%%%%%%%%%%%%%%%%%%%%%%%%%%%%%%%%%%%%%%%%%%%%
We know that
%%%%%%%%%%%%%%%%%%%%%%%%%%%%%%%%%%%%%%%%%%%%%%
\begin{equation}\label{4}
\left(a,b\right) = p_{1}^{\wedge\left(r_1,s_1\right)}p_{2}^{\wedge\left(r_2,s_2\right)}\cdots p_{t}^{\wedge\left(r_t,s_t\right)}
\end{equation}
%%%%%%%%%%%%%%%%%%%%%%%%%%%%%%%%%%%%%%%%%%%%%%
\begin{equation*}
\Lcm\left(a,b\right) = p_{1}^{\vee\left(r_1,s_1\right)}p_{2}^{\vee\left(r_2,s_2\right)}\cdots p_{t}^{\vee\left(r_t,s_t\right)}
\end{equation*} 
%%%%%%%%%%%%%%%%%%%%%%%%%%%%%%%%%%%%%%%%%%%%%%
where $\wedge(r_i,s_i)$ and $\vee(r_i,s_i)$ denote the minimum and maximum of $r_i$ and $s_i$ respectively. Using this fact, we can rewrite the equation as
%%%%%%%%%%%%%%%%%%%%%%%%%%%%%%%%%%%%%%%%%%%%%%
\begin{equation*}
I\left(x,p_{1}^{\wedge\left( r_1, s_1\right)}\right)I\left(x,p_{1}^{\vee\left(r_1,s_1\right)}\right)I\left(x,p_{2}^{\wedge\left( r_2, s_2\right)}\right)I\left(x,p_{2}^{\vee\left(r_2,s_2\right)}\right) \cdots I\left(x,p_{m}^{\wedge\left( r_m, s_m\right)}\right)I\left(x,p_{m}^{\vee\left(r_m,s_m\right)}\right)
\end{equation*} 
%%%%%%%%%%%%%%%%%%%%%%%%%%%%%%%%%%%%%%%%%%%%%%
\begin{equation*}
= I\left(x,p_{1}^{\wedge\left( r_1, s_1\right)}p_{1}^{\wedge\left( r_2, s_2\right)}\cdots p_{m}^{\wedge\left( r_m, s_m\right)}\right)I\left(x,p_{1}^{\vee\left( r_1, s_1\right)}p_{1}^{\vee\left( r_2, s_2\right)}\cdots p_{m}^{\vee\left( r_m, s_m\right)}\right).
\end{equation*}
%%%%%%%%%%%%%%%%%%%%%%%%%%%%%%%%%%%%%%%%%%%%%%
From equation (\ref{4}), 
%%%%%%%%%%%%%%%%%%%%%%%%%%%%%%%%%%%%%%%%%%%%%%
\begin{equation*} 
I\left(x,a\right)I\left(x,b\right)  = I\left(x,\left(a,b\right)\right)I\left(x,\Lcm\left(a,b\right)\right).
\end{equation*} 
%%%%%%%%%%%%%%%%%%%%%%%%%%%%%%%%%%%%%%%%%%%%%%
\end{proof}
%%%%%%%%%%%%%%%%%%%%%%%%%%%%%%%%%%%%%%%%%%%%%%

%%%%%%%%%%%%%%%%%%%%%%%%%%%%%%%%%%%%%%%%%%%%%%
\section{Limiting Properties and Bounds on the \\ \xth Abundancy Index}
%%%%%%%%%%%%%%%%%%%%%%%%%%%%%%%%%%%%%%%%%%%%%%
Here we analyze limiting properties and upper bounds on the function $I(x,n)$ and improve previously known results. The following proposition is a generalized version of a theorem used in \cite{Laatsch}. We will make great use of the result when identifying \xth abundancy outlaws. 
%%%%%%%%%%%%%%%%%%%%%%%%%%%%%%%%%%%%%%%%%%%%%%
\begin{prop} \label{Proposition1}
Let $n$ and $k$ be positive integers. If $k>1$, then  $I\left(x,kn\right)>I\left(x,n\right)$.
\end{prop}
%%%%%%%%%%%%%%%%%%%%%%%%%%%%%%%%%%%%%%%%%%%%%%
\begin{proof}
Let $n$ and $k$ be positive integers. If $1, a_0, a_1, a_2,\ldots, a_s, n$ are the divisors of $n$, then $1, k, ka_0, ka_1, ka_2,\ldots, ka_s, kn$ is a set of divisors of $kn$. We can bound $I\left(x,kn\right)$ by
%%%%%%%%%%%%%%%%%%%%%%%%%%%%%%%%%%%%%%%%%%%%%%
\begin{align*}
I\left(x,kn\right) &\geqslant \frac{1+(k)^x+(ka_0)^x+(ka_1)^x+(ka_2)^x+\cdots+(ka_s)^x+(kn)^x}{(kn)^x} \\
&\geqslant \frac{1}{(kn)^x}+ \left(\frac{k^x \left( 1+ a_0^x+a_1^x+a_2^x+\cdots+a_s^x+n^x \right)}{(kn)^x}\right) \\  
& \geqslant \frac{1}{(kn)^x}  + I\left(x,n\right)>I\left(x,n\right).
\end{align*}
%%%%%%%%%%%%%%%%%%%%%%%%%%%%%%%%%%%%%%%%%%%%%%
Therefore, $I\left(x,kn\right)>I\left(x,n\right)$.
\end{proof}
%%%%%%%%%%%%%%%%%%%%%%%%%%%%%%%%%%%%%%%%%%%%%%
From Proposition \ref{Proposition1}, we see that the \xth abundancy index of any multiple of a positive integer increases. Our next goal is to extend upper bound properties regarding prime powers. We improve Property \ref{p1} and Property \ref{p2} by categorizing positive integers of the form $nm^k$, where $n$, $m$ are positive integers and $k$ a nonnegative integer. By doing so, we can find $\displaystyle\lim_{k \to \infty}I\left(x,nm^k\right)$ for any $n$ and $m$. We first observe cases where $(n,m)=1$. Building off limiting values and bounds on the \xth abundancy index, we take note of ratio properties using the $nm^k$ categorization. 
%%%%%%%%%%%%%%%%%%%%%%%%%%%%%%%%%%%%%%%%%%%%%%
\begin{prop}\label{a1}
Let $n_1$, $n_2$, $m$, $k$ be positive integers and $j$ a nonnegative integer. If $\left(n_1,m\right)=1$ and $\left(n_2,m\right)=1$, then
\begin{equation*}
\frac{I\left(x,n_1 m^k\right)}{I\left(x,n_1 m^j\right)} = \frac{I\left(x,n_2 m^k\right)}{I\left(x,n_2 m^j\right)}.
\end{equation*}
\end{prop}
%%%%%%%%%%%%%%%%%%%%%%%%%%%%%%%%%%%%%%%%%%%%%%
\begin{proof}
%%%%%%%%%%%%%%%%%%%%%%%%%%%%%%%%%%%%%%%%%%%%%%
Let $n_1$, $n_2$, $m$, $k$ be positive integers and $j$ a nonnegative integer. Since $I(x,n)$ multiplicative, we have
%%%%%%%%%%%%%%%%%%%%%%%%%%%%%%%%%%%%%%%%%%%%%%
\[
\frac{I\left(x,n_1 m^k\right)}{I\left(x,n_1 m^j\right)} = \frac{I\left(x,  m^k\right)}{I\left(x, m^j\right)} = \frac{I\left(x,n_2 m^k\right)}{I\left(x,n_2 m^j\right)}.
\]
%%%%%%%%%%%%%%%%%%%%%%%%%%%%%%%%%%%%%%%%%%%%%%
\end{proof}
%%%%%%%%%%%%%%%%%%%%%%%%%%%%%%%%%%%%%%%%%%%%%%
From Proposition \ref{a1}, we notice that ratios of the \xth abundancy index remain constant when $m$ is fixed. Next we take a look at the limiting value for any positive integer power. 
%%%%%%%%%%%%%%%%%%%%%%%%%%%%%%%%%%%%%%%%%%%%%%
\begin{prop}\label{Proposition2}
If $m$ is a positive integer and $k$ a nonnegative integer with $m$ having the prime factorization $m=p_1^{s_1}p_2^{s_2}\cdots p_t^{s_t}$, then 
\[
\displaystyle{\lim_{k \to \infty}}I\left(x,m^k\right) =  \prod_{i=1}^{t} \frac{p_i^x}{p_i^x-1}.
\]
\end{prop}
%%%%%%%%%%%%%%%%%%%%%%%%%%%%%%%%%%%%%%%%%%%%%%
\begin{proof}
%%%%%%%%%%%%%%%%%%%%%%%%%%%%%%%%%%%%%%%%%%%%%%
Let $m$ be a positive integer and $k$ a nonnegative integer with $m$ having the prime factorization $m=p_1^{s_1}p_2^{s_2}\cdots p_t^{s_t}$. Since $I(x,n)$ multiplicative, 
%%%%%%%%%%%%%%%%%%%%%%%%%%%%%%%%%%%%%%%%%%%%%%
\begin{align*}
\lim_{k \to \infty} I\left(x,m^k \right) 
%%%%%%%%%%%%%%%%%%%%%%%%%%%%%%%%%%%%%%%%%%%%%%
& =   \lim_{k \to \infty} I\left(x, \left( p_1^{s_1} p_2^{s_2}\cdots  p_t^{s_t} \right)^k \right) \\ 
%%%%%%%%%%%%%%%%%%%%%%%%%%%%%%%%%%%%%%%%%%%%%%
& =   \lim_{k \to \infty} I\left(x, p_1^{ks_1} \right) \cdots I\left(x, p_t^{ks_t} \right).
\end{align*}
%%%%%%%%%%%%%%%%%%%%%%%%%%%%%%%%%%%%%%%%%%%%%%
By the definition of $I(x,n)$, 
%%%%%%%%%%%%%%%%%%%%%%%%%%%%%%%%%%%%%%%%%%%%%%
\[
\lim_{k \to \infty} I\left(x, p_i^{k s_i} \right) = \lim_{k \to \infty} \frac{ \sum_{j=0}^{k s_i} p_i^{xj} }{ p_i^{ x k s_i } } 
= \lim_{k \to \infty} \sum_{j=0}^{ks_i} \left( \frac{1}{p_i} \right)^{x(ks_i -j)}
\] 
%%%%%%%%%%%%%%%%%%%%%%%%%%%%%%%%%%%%%%%%%%%%%%
where $1\leq i\leq t$. Using a geometric sum, we can rewrite the equation as 
%%%%%%%%%%%%%%%%%%%%%%%%%%%%%%%%%%%%%%%%%%%%%%
\begin{align*} 
\lim_{k \to \infty}  \sum_{j=0}^{ks_i} \left(\frac{1}{p_i}\right)^{x(ks_i -j)} 
& =  \lim_{k \to \infty} \sum_{j=0}^{ks_i} \left(\frac{1}{p_i}\right)^{xj} \\
%%%%%%%%%%%%%%%%%%%%%%%%%%%%%%%%%%%%%%%%%%%%%%
& = \sum_{j=0}^{\infty} \left( \frac{1}{p_i^x} \right)^j  \\
%%%%%%%%%%%%%%%%%%%%%%%%%%%%%%%%%%%%%%%%%%%%%%
& = \left(\frac{1}{1-\frac{1}{p_i^x}}\right) = \frac{p_i^x}{p_i^x-1}.
%%%%%%%%%%%%%%%%%%%%%%%%%%%%%%%%%%%%%%%%%%%%%%
\end{align*} 
%%%%%%%%%%%%%%%%%%%%%%%%%%%%%%%%%%%%%%%%%%%%%%
Therefore,
%%%%%%%%%%%%%%%%%%%%%%%%%%%%%%%%%%%%%%%%%%%%%%
\[
\lim_{k \to \infty} I\left(x,m^k\right) = \ \prod_{i=1}^{t} \frac{p_i^x}{p_i^x-1}.
\]
%%%%%%%%%%%%%%%%%%%%%%%%%%%%%%%%%%%%%%%%%%%%%%
\end{proof}
%%%%%%%%%%%%%%%%%%%%%%%%%%%%%%%%%%%%%%%%%%%%%%
Using the previous two propositions, we look at cases where $n$ and $m$ are not coprime. In these cases, limiting values and ratios of the \xth abundancy index become more intricate. 
%%%%%%%%%%%%%%%%%%%%%%%%%%%%%%%%%%%%%%%%%%%%%%
\begin{prop}
Let $n_1$, $n_2$, $m$, $k$ be positive integers and $j$ a nonnegative integer. If $n_1 = a_1b$ and $n_2=a_2b$ where $a_1$ and $a_2$ are the largest divisors of $n_1$ and $n_2$, respectively, such that $\left(a_1,m\right)=1$ and $\left(a_2,m\right)=1$, then
%%%%%%%%%%%%%%%%%%%%%%%%%%%%%%%%%%%%%%%%%%%%%%
\begin{equation*}
\frac{I\left(x,n_1m^k\right)}{I\left(x,n_1m^j\right)} = \frac{I\left(x,n_2m^k\right)}{I\left(x,n_2m^j\right)}.
\end{equation*} 
%%%%%%%%%%%%%%%%%%%%%%%%%%%%%%%%%%%%%%%%%%%%%%
\end{prop}
%%%%%%%%%%%%%%%%%%%%%%%%%%%%%%%%%%%%%%%%%%%%%%
\begin{proof}
Let $n_1$, $n_2$, $m$, $k$ be positive integers and $j$ a nonnegative integer. We can substitute $a_1b$ for $n_1$ to get
%%%%%%%%%%%%%%%%%%%%%%%%%%%%%%%%%%%%%%%%%%%%%%
\begin{equation*}
\frac{I\left(x,n_1m^k\right)}{I\left(x,n_1m^j\right)} = \frac{I\left(x,a_1bm^k\right)}{I\left(x,a_1bm^j\right)}.
\end{equation*}
%%%%%%%%%%%%%%%%%%%%%%%%%%%%%%%%%%%%%%%%%%%%%%
Since $I(x,n)$ is multplicative,  
%%%%%%%%%%%%%%%%%%%%%%%%%%%%%%%%%%%%%%%%%%%%%%
\begin{equation*}
\frac{I\left(x,a_1bm^k\right)}{I\left(x,a_1bm^j\right)} =\frac{I\left(x,bm^k\right)}{I\left(x,bm^j\right)} = \frac{I\left(x,a_2bm^k\right)}{I\left(x,a_2bm^j\right)}=\frac{I\left(x,n_2m^k\right)}{I\left(x,n_2m^j\right)}.
\end{equation*}
%%%%%%%%%%%%%%%%%%%%%%%%%%%%%%%%%%%%%%%%%%%%%%
\end{proof}
%%%%%%%%%%%%%%%%%%%%%%%%%%%%%%%%%%%%%%%%%%%%%%
\d*
%%%%%%%%%%%%%%%%%%%%%%%%%%%%%%%%%%%%%%%%%%%%%%
\begin{proof}
%%%%%%%%%%%%%%%%%%%%%%%%%%%%%%%%%%%%%%%%%%%%%%
Let $n$ and $m$ be positive integers and $k$ a nonnegative integer with $m$ having the prime factorization $m=p_1^{s_1}p_2^{s_2}\cdots p_t^{s_t}$ and $n=ab$, where $a$ is the largest divisor of $n$ such that $\left(a,m\right)=1$. We begin by substituting $ab$ for $n$ and using the multiplicative properties of $I(x,n)$ to obtain 
%%%%%%%%%%%%%%%%%%%%%%%%%%%%%%%%%%%%%%%%%%%%%%
\begin{align*}
\displaystyle{\lim_{k \to \infty}}I\left(x,nm^k\right) & = \displaystyle{\lim_{k \to \infty}}I\left(x,abm^k\right)\\& =I\left(x,a\right)\displaystyle{\lim_{k \to \infty}}I\left(x,bm^k\right).
\end{align*}
%%%%%%%%%%%%%%%%%%%%%%%%%%%%%%%%%%%%%%%%%%%%%%
By the definition of $b$, we know $b$ must have the prime factorization $b=p_1^{c_1}p_2^{c_2}\cdots p_t^{c_t}$, where $c_i$ is nonnegative and $c_i\leq s_i$ for all $1\leq i \leq t$. This gives us 
%%%%%%%%%%%%%%%%%%%%%%%%%%%%%%%%%%%%%%%%%%%%%%
\begin{align*}
I\left(x,a\right)\displaystyle{\lim_{k \to \infty}}I\left(x,b\left(p_1^{ks_1}p_2^{ks_2}\cdots p_t^{ks_t}\right)\right) 
&=  I\left(x,a\right)\displaystyle{\lim_{k \to \infty}}I\left(x,\left(p_1^{c_1}p_2^{c_2}\cdots p_t^{c_t}\right)\left(p_1^{ks_1}p_2^{ks_2}\cdots p_t^{ks_t}\right)\right)\\
& = I\left(x,a\right)\displaystyle{\lim_{k \to \infty}}I\left(x,p_1^{ks_1+c_1}\right)I\left(x,p_2^{ks_2+c_2}\right)\cdots I\left(x,p_t^{ks_t+c_t}\right).
\end{align*}
%%%%%%%%%%%%%%%%%%%%%%%%%%%%%%%%%%%%%%%%%%%%%%
We have that as $k$ approaches infinity, $k s_i+c_i$ approaches infinity for all $1\leq i \leq t$. From Proposition \ref{Proposition2}, 
%%%%%%%%%%%%%%%%%%%%%%%%%%%%%%%%%%%%%%%%%%%%%%
\[
\lim_{k \to \infty} I\left(x,nm^k\right) =  I\left(x,a\right) \ \prod_{i=1}^{t} \frac{p_i^x}{p_i^x-1}.
\]
%%%%%%%%%%%%%%%%%%%%%%%%%%%%%%%%%%%%%%%%%%%%%%
\end{proof}
%%%%%%%%%%%%%%%%%%%%%%%%%%%%%%%%%%%%%%%%%%%%%%
Proposition \ref{PropositionPerfect} gives the limiting value for any combination of positive integers under the \xth abundancy index. Returning to Theorem \ref{a2}, we know that even perfect numbers are of the form $N =2^{p-1}(2^p -1)$, where $p$ and $(2^p -1)$ are primes. Using Proposition \ref{PropositionPerfect}, we obtain the following proposition dealing with positive integers that share the same form with even perfect numbers. 
%%%%%%%%%%%%%%%%%%%%%%%%%%%%%%%%%%%%%%%%%%%%%%
\begin{prop} \label{c1}
Let $p_1,p_2,\ldots, p_k$ be the sequence of prime numbers in increasing order. Consider the sequence of numbers denoted by $N_1,N_2,\ldots, N_k$, where $N_i=2^{p_i-1}(2^{p_i} -1)$ for $1\leq i \leq k$. Then
\begin{equation*}
\displaystyle{\lim_{k \to \infty}}I\left(x,N_k\right) = \frac{2^x}{2^x-1}.
\end{equation*}
\end{prop}
%%%%%%%%%%%%%%%%%%%%%%%%%%%%%%%%%%%%%%%%%%%
\begin{proof}
%%%%%%%%%%%%%%%%%%%%%%%%%%%%%%%%%%%%%%%%%%%
We begin by substituting $2^{p_k-1}\left(2^{p_k}-1\right)$ for $N_k$, 
%%%%%%%%%%%%%%%%%%%%%%%%%%%%%%%%%%%%%%%%%%%
\begin{align*}
\displaystyle{\lim_{k \to \infty}}I\left(x,N_k\right) &= \displaystyle{\lim_{k \to \infty}}I\left(x,2^{p_k-1}\left(2^{p_k}-1\right)\right) \\ &=\displaystyle{\lim_{k \to \infty}}I\left(x,2^{p_k-1}\right)I\left(x,2^{p_k}-1\right) . 
\end{align*}
%%%%%%%%%%%%%%%%%%%%%%%%%%%%%%%%%%%%%%%%%%%
Proposition \ref{PropositionPerfect} gives us 
%%%%%%%%%%%%%%%%%%%%%%%%%%%%%%%%%%%%%%%%%%%
\begin{equation*}
\displaystyle{\lim_{k \to \infty}}I\left(x,2^{p_k-1}\right) = \frac{2^x}{2^x-1}. 
\end{equation*}
%%%%%%%%%%%%%%%%%%%%%%%%%%%%%%%%%%%%%%%%%%%
Since $2^{p_k}-1$ is a prime number, 
%%%%%%%%%%%%%%%%%%%%%%%%%%%%%%%%%%%%%%%%%%%
\begin{equation*}
\displaystyle{\lim_{k \to \infty}}I\left(x,2^{p_k}-1\right) = \displaystyle{\lim_{k \to \infty}}\frac{\left(2^{p_k}-1\right)^x + 1}{\left(2^{p_k}-1\right)^x} = 1.
\end{equation*}
%%%%%%%%%%%%%%%%%%%%%%%%%%%%%%%%%%%%%%%%%%%
Collecting the pieces, 
%%%%%%%%%%%%%%%%%%%%%%%%%%%%%%%%%%%%%%%%%%%
\begin{equation*}
\displaystyle{\lim_{k \to \infty}}I\left(x,N_k\right) =\frac{2^x}{2^x-1}.
\end{equation*}
%%%%%%%%%%%%%%%%%%%%%%%%%%%%%%%%%%%%%%%%%%%
\end{proof}
%%%%%%%%%%%%%%%%%%%%%%%%%%%%%%%%%%%%%%%%%%%
Proposition \ref{c1} proves to be an interesting result as we notice that positive integers of this form, particularly even perfect numbers, approach $\frac{2^x}{2^x-1}$ under the \xth abundancy index. Knowing this fact, we can predict the limiting value of even perfect numbers under the \xth abundancy index as they grow larger, if infinitely many do exist. 
%%%%%%%%%%%%%%%%%%%%%%%%%%%%%%%%%%%%%%%%%%%%%%

%%%%%%%%%%%%%%%%%%%%%%%%%%%%%%%%%%%%%%%%%%%%%%
\section{\xth Abundancy Outlaws}
%%%%%%%%%%%%%%%%%%%%%%%%%%%%%%%%%%%%%%%%%%%%%%
We now focus on generalizing properties of abundancy outlaws as \xth abundancy outlaws. Our goal is to determine which rationals greater than one fail to be in the image of the function $I(x,n)$. The following properties will be extremely useful in doing so \cite{Holdener}.
%%%%%%%%%%%%%%%%%%%%%%%%%%%%%%%%%%%%%%%%%%%%%%
\begin{pro}  \label{Property1}
Let $n$, $m$, and $k$ be positive integers. If $I(n) = \frac{k}{m}$ with $(k, m) = 1$, then $m$ divides $n$.
\end{pro}
%%%%%%%%%%%%%%%%%%%%%%%%%%%%%%%%%%%%%%%%%%%%%%
\begin{pro}  \label{Property2}
Let $n$, $m$, and $x$ be positive integers. Then $m^x$ divides $n^x$ if and only if $m$ divides $n$.
\end{pro}
%%%%%%%%%%%%%%%%%%%%%%%%%%%%%%%%%%%%%%%%%%%%%%
\begin{proof}
%%%%%%%%%%%%%%%%%%%%%%%%%%%%%%%%%%%%%%%%%%%%%%
This follows directly from the Fundamental Theorem of Arithmetic.
%%%%%%%%%%%%%%%%%%%%%%%%%%%%%%%%%%%%%%%%%%%%%%
\end{proof}
%%%%%%%%%%%%%%%%%%%%%%%%%%%%%%%%%%%%%%%%%%%%%%
Using these two properties and Proposition \ref{Proposition1}, we move on to our main results.

%%%%%%%%%%%%%%%%%%%%%%%%%%%%%%%%%%%%%%%%%%%%%%
\a*
%%%%%%%%%%%%%%%%%%%%%%%%%%%%%%%%%%%%%%%%%%%%%%
\begin{proof}
%%%%%%%%%%%%%%%%%%%%%%%%%%%%%%%%%%%%%%%%%%%%%%
Let $m$ and $k$ be positive integers. For sake of contradiction, suppose $\frac{k}{m^x}$ is an \xth abundancy index. It follows that $I\left(x,n\right) = \frac{k}{m^x}$ for some positive integer $n$ and 
%%%%%%%%%%%%%%%%%%%%%%%%%%%%%%%%%%%%%%%%%%%%%%
\begin{equation*}
m^x\sigma_x(n) = kn^x. 
\end{equation*}
%%%%%%%%%%%%%%%%%%%%%%%%%%%%%%%%%%%%%%%%%%%%%%
By Properties \ref{Property1} and \ref{Property2}, $m$ divides $n$. From Proposition \ref{Proposition1}, $I\left(x,n\right)>I\left(x,m\right)$, hence, 
%%%%%%%%%%%%%%%%%%%%%%%%%%%%%%%%%%%%%%%%%%%%%%
\begin{align*}
&\frac{\sigma_x(m)}{m^x} < \frac{\sigma_x(n)}{n^x}=\frac{k}{m^x}.
\end{align*}
%%%%%%%%%%%%%%%%%%%%%%%%%%%%%%%%%%%%%%%%%%%%%%
Therefore, we have a contradiciton as $\sigma_x(m)< k$, making $\frac{k}{m^x}$ an \xth abundancy outlaw.
%%%%%%%%%%%%%%%%%%%%%%%%%%%%%%%%%%%%%%%%%%%%%%
\end{proof}
%%%%%%%%%%%%%%%%%%%%%%%%%%%%%%%%%%%%%%%%%%%%%%
Theorem \ref{Theorem1} generates a class of \xth abundancy outlaws of the form $\frac{\sigma_x(n)-t}{n^x}$, where $t$ is a positive integer. Next we generalize Holdener's and Stanton's work \cite{Holdener}. We first extend Theorem \ref{Theorem1} by locating \xth abundancy outlaws of a similar form $\frac{\sigma_x(n)+t}{n^x}$, where $t$ is a positive integer. The following lemma gives an important inequality we use when finding these \xth abundancy outlaws. 
%%%%%%%%%%%%%%%%%%%%%%%%%%%%%%%%%%%%%%%%%%%%%%

%%%%%%%%%%%%%%%%%%%%%%%%%%%%%%%%%%%%%%%%%%%%%%
\begin{lem} \label{Lemma2}
Let $n$ be a positive integer with $n= \prod_{i=1}^{s} p_i^{k_i}$ for primes $p_1, p_2,\ldots, p_s$. For a given $p_j$ where $1\leq j\leq s$ and a positive integer $t$, 
%%%%%%%%%%%%%%%%%%%%%%%%%%%%%%%%%%%%%%%%%%%%%%
\begin{align*}
\frac{\sigma_x(n)+t}{n^x} < I(x,p_jn) \text{ if and only if  } \ p_j^x<\frac{1}{t}\sigma_x\left(\frac{n}{p_j^{k_j}}\right).
\end{align*}
%%%%%%%%%%%%%%%%%%%%%%%%%%%%%%%%%%%%%%%%%%%%%%
\end{lem}
%%%%%%%%%%%%%%%%%%%%%%%%%%%%%%%%%%%%%%%%%%%%%%
\begin{proof}
%%%%%%%%%%%%%%%%%%%%%%%%%%%%%%%%%%%%%%%%%%%%%%
Let $n$ be a positive integer with $n= \prod_{i=1}^{s} p_i^{k_i}$ for primes $p_1, p_2,\ldots, p_s$. For a given $p_j$ where $1\leq j\leq s$ and a positive integer $t$, suppose 
%%%%%%%%%%%%%%%%%%%%%%%%%%%%%%%%%%%%%%%%%%%%%%
\begin{align*}
\frac{\sigma_x(n)+t}{n^x} < I(x,p_jn).
\end{align*}
%%%%%%%%%%%%%%%%%%%%%%%%%%%%%%%%%%%%%%%%%%%%%%
This implies
%%%%%%%%%%%%%%%%%%%%%%%%%%%%%%%%%%%%%%%%%%%%%%
\begin{align*}
p_j^x\sigma_x(n) + p_j^xt < \sigma_x(p_jn).
\end{align*}
%%%%%%%%%%%%%%%%%%%%%%%%%%%%%%%%%%%%%%%%%%%%%%
Examining the left hand side of the inequality, 
%%%%%%%%%%%%%%%%%%%%%%%%%%%%%%%%%%%%%%%%%%%%%%
\begin{align*} 
p_j^x\sigma_x(n) + p_j^xt = p_j^x\sigma_x\left(p_j^{k_j}\right)\sigma_x\left(\frac{n}{p_j^{k_j}}\right) + p_j^xt = \left(\sigma_x\left(p_j^{k_j+1}\right)-1\right)\sigma_x\left(\frac{n}{p_j^{k_j}}\right) + p_j^xt.
\end{align*}
%%%%%%%%%%%%%%%%%%%%%%%%%%%%%%%%%%%%%%%%%%%%%%
From here we have that 
%%%%%%%%%%%%%%%%%%%%%%%%%%%%%%%%%%%%%%%%%%%%%%
\begin{equation*}
p_j^x<\frac{1}{t}\sigma_x\left(\frac{n}{p_j^{k_j}}\right).
\end{equation*}
%%%%%%%%%%%%%%%%%%%%%%%%%%%%%%%%%%%%%%%%%%%%%%
Conversely, suppose $p_j^x<\frac{1}{t}\sigma_x\left(\frac{n}{p_j^{k_j}}\right)$. Using the same argument, we show that $\frac{\sigma_x(n)+t}{n^x} < I(x,p_jn)$. Therefore, 
%%%%%%%%%%%%%%%%%%%%%%%%%%%%%%%%%%%%%%%%%%%%%%
\begin{align*}
\frac{\sigma_x(n)+t}{n^x} < I(x,p_jn) \text{ if and only if  } \ p_j^x<\frac{1}{t}\sigma_x\left(\frac{n}{p_j^{k_j}}\right).
\end{align*}
%%%%%%%%%%%%%%%%%%%%%%%%%%%%%%%%%%%%%%%%%%%%%%
\end{proof}
%%%%%%%%%%%%%%%%%%%%%%%%%%%%%%%%%%%%%%%%%%%%%%

%%%%%%%%%%%%%%%%%%%%%%%%%%%%%%%%%%%%%%%%%%%%%%
\b*
%%%%%%%%%%%%%%%%%%%%%%%%%%%%%%%%%%%%%%%%%%%%%%
\begin{proof}
%%%%%%%%%%%%%%%%%%%%%%%%%%%%%%%%%%%%%%%%%%%%%%
Case 1: For a positive integer $t$, let $\frac{\sigma_x(n)+t}{n^x}$  be a fraction such that $\left(\sigma_x(n)+t, n^x\right)=1$, and let  $n^x= \prod_{i=1}^{s} p_i^{xk_i}$. Suppose that there exists a positive integer $1\leq j \leq s$ such that $p_j^x<\frac{1}{t}\sigma_x\left(\frac{n}{p_j^{k_j}}\right)$ and suppose further that $\sigma_x\left(p_j^{k_j}\right)$ has a divisor $d^x $ greater than one such that $I\left(x,p_j^{k_j}\right)I(x,d) > \frac{\sigma_x(n)+t}{n^x}$ and $(d^x,t)=1$. For sake of contradiction, suppose that $I(x,a) = \frac{\sigma_x(n)+t}{n^x}$ for some positive integer $a$. Using Properties \ref{Property1} and \ref{Property2}, $n$ divides $a$, which gives us $a=mn$ for some integer $m$. From our initial assumption, $p_j^x<\frac{1}{t}\sigma_x\left(\frac{n}{p_j^{k_j}}\right)$. By Lemma \ref{Lemma2}, 
%%%%%%%%%%%%%%%%%%%%%%%%%%%%%%%%%%%%%%%%%%%%%%
\begin{align*} 
I(x,a)=\frac{\sigma_x(n)+t}{n^x}<I(x,p_jn),
\end{align*}
%%%%%%%%%%%%%%%%%%%%%%%%%%%%%%%%%%%%%%%%%%%%%%
and hence $p_j^{k_j+1}$ does not divide $a$, meaning $p_j$ does not divide $m$. We can rewrite $I(x,mn)$ as $I\left(x,p_j^{k_j}\cdot\frac{mn}{p_j^{k_j}}\right)$ and because $I(x,n)$ is multiplicative, 
%%%%%%%%%%%%%%%%%%%%%%%%%%%%%%%%%%%%%%%%%%%%%%
\begin{align*}
I(x,a)=I\left(x,p_j^{k_j}\right)I\left(x,\frac{mn}{p_j^{k_j}}\right)=\frac{\sigma_x(n)+t}{n^x},
\end{align*}  
%%%%%%%%%%%%%%%%%%%%%%%%%%%%%%%%%%%%%%%%%%%%%%
this implies
%%%%%%%%%%%%%%%%%%%%%%%%%%%%%%%%%%%%%%%%%%%%%%
\begin{align*}
\sigma_x\left(p_j^{k_j}\right)\sigma_x\left(\frac{mn}{p_j^{k_j}}\right) = (\sigma_x(n)+t)m^x.
\end{align*}
%%%%%%%%%%%%%%%%%%%%%%%%%%%%%%%%%%%%%%%%%%%%%%
Combining our initial assumption that $(d^x,t)=1$ and $d^x$ divides $\sigma_x\left(p_j^{k_j}\right)$, this implies $\left(d^x,\sigma_x(n)+t\right)=1$. Hence, $d^x$ divides $(\sigma_x(n)+t)m^x$ implies $d^x$ divides $m^x$. By Property \ref{Property2}, $d$ divides $m$, giving $d$ divides $\left(\frac{mn}{p_j^{k_j}}\right)$.
%%%%%%%%%%%%%%%%%%%%%%%%%%%%%%%%%%%%%%%%%%%%%%
Using Proposition \ref{Proposition1}, 
%%%%%%%%%%%%%%%%%%%%%%%%%%%%%%%%%%%%%%%%%%%%%%
\begin{align*}
&I\left(x,p_j^{k_j}\right)I(x,d) < I\left(x,p_j^{k_j}\right) I\left(x,\frac{mn}{p_j^{k_j}}\right)=I(x,a)=\frac{\sigma_x(n)+t}{n^x} 
\end{align*}
%%%%%%%%%%%%%%%%%%%%%%%%%%%%%%%%%%%%%%%%%%%%%%
which implies
%%%%%%%%%%%%%%%%%%%%%%%%%%%%%%%%%%%%%%%%%%%%%%
\begin{align*}
 I\left(x,p_j^{k_j}\right)I(x,d) \leq \frac{\sigma_x(n)+t}{n^x}.
\end{align*}
%%%%%%%%%%%%%%%%%%%%%%%%%%%%%%%%%%%%%%%%%%%%%%
Therefore, we have a contradiction and $\frac{\sigma_x(n)+t}{n^x}$ is an \xth abundancy outlaw. \\ \\ 
%%%%%%%%%%%%%%%%%%%%%%%%%%%%%%%%%%%%%%%%%%%%%%

%%%%%%%%%%%%%%%%%%%%%%%%%%%%%%%%%%%%%%%%%%%%%%
Case 2: For a positive integer $t$, let $\frac{\sigma_x(n)+t}{n^x}$  be a fraction such that $\left(\sigma_x(n)+t, n^x\right)=1$, and let  $n^x= \prod_{i=1}^{s} p_i^{xk_i}$. Suppose that there exists a positive integer $1\leq j \leq s$ such that $p_j^x<\frac{1}{t}\sigma_x\left(\frac{n}{p_j^{k_j}}\right)$ and suppose further that $\sigma_x\left(p_j^{k_j}\right)$ has a divisor $d^x $ greater than one such that $(d^x, n^xt) = 1$. For sake of contradiction, suppose that $I(x,a) = \frac{\sigma_x(n)+t}{n^x}$ for some positive integer $a$. From Properties \ref{Property1} and \ref{Property2}, $n$ divides $a$, which gives us $a=mn$ for some integer $m$. Using Lemma \ref{Lemma2}, 
%%%%%%%%%%%%%%%%%%%%%%%%%%%%%%%%%%%%%%%%%%%%%%
\begin{align*} 
&I(x,a)=\frac{\sigma_x(n)+t}{n^x}<I(x,p_jn) 
\end{align*}
%%%%%%%%%%%%%%%%%%%%%%%%%%%%%%%%%%%%%%%%%%%%%%
implying $p_j$ does not divide $m$. Since $I(x,n)$ is multiplicative, 
%%%%%%%%%%%%%%%%%%%%%%%%%%%%%%%%%%%%%%%%%%%%%%
\begin{align*}
 I(x,a)& =I\left(x,p_j^{k_j}\cdot m\frac{n}{p_j^{k_j}}\right) \\
& =I\left(x,p_j^{k_j}\right)I\left(x,m\frac{n}{p_j^{k_j}}\right). 
\end{align*}
%%%%%%%%%%%%%%%%%%%%%%%%%%%%%%%%%%%%%%%%%%%%%%
Let $\left(m,\frac{n}{p_j^{k_j}}\right)=\prod_{i=1}^{r} p_i^{q_i}$, we can set $m_0$ as
%%%%%%%%%%%%%%%%%%%%%%%%%%%%%%%%%%%%%%%%%%%%%%
\begin{align*}
m_0 = \frac{m}{\prod_{i=1}^{r} p_i^{q_i}}
\end{align*}
%%%%%%%%%%%%%%%%%%%%%%%%%%%%%%%%%%%%%%%%%%%%%%
where $(m_0,\frac{m}{\prod_{i=1}^{r} p_i^{q_i}})=1$. Since $I(x,n)$ is multiplicative,
%%%%%%%%%%%%%%%%%%%%%%%%%%%%%%%%%%%%%%%%%%%%%%
\begin{align*}
I(x,a)= I\left(x,p_j^{k_j}\right)I\left(x,m_0\right)I\left(x,\frac{n}{p_j^{k_j}}\prod_{i=1}^{r} p_i^{q_i}\right)=\frac{\sigma_x(n)+t}{n^x}.
\end{align*}
%%%%%%%%%%%%%%%%%%%%%%%%%%%%%%%%%%%%%%%%%%%%%%
We can rewrite the equation as 
%%%%%%%%%%%%%%%%%%%%%%%%%%%%%%%%%%%%%%%%%%%%%%
\begin{align}\label{Equation1}
 \sigma_x\left(p_j^{k_j}\right)\sigma_x\left(m_0\right)\sigma_x\left(\frac{n}{p_j^{k_j}}\prod_{i=1}^{r} p_i^{q_i}\right)= (\sigma_x(n)+t)m_0^x\prod_{i=1}^{r} p_i^{xq_i}.
\end{align}
%%%%%%%%%%%%%%%%%%%%%%%%%%%%%%%%%%%%%%%%%%%%%%
Because $d^x$ divides $\sigma_x(p_i^{k_i})$ and $\sigma_x(p_i^{k_i})$ divides $\sigma_x(n)$, $d^x$ divides $\sigma_x(n)$. Combining this with our initial assumption $(d^x, n^xt) = 1$, this implies $(d^x,t)=1$, hence, $d^x$ does not divide $\sigma_x(n)+t$. From (\ref{Equation1}), $d^x$ divides $m_0^x\prod_{i=1}^{r} p_i^{xq_i}$. Returning to the fact that $(d^x, n^xt) = 1$, we know that $(d^x,n^x)=1$, implying no prime power factor $p_i$ of $n$ divides $d^x$. Thus, $d^x$ divides $m_0^x$. By Property \ref{Property2} and Proposition \ref{Proposition1}, $I(x,m_0)>I(x,d)$. From our initial assumption, $p_j^x<\frac{1}{t}\sigma_x\left(\frac{n}{p_j^{k_j}}\right)$, this implies
%%%%%%%%%%%%%%%%%%%%%%%%%%%%%%%%%%%%%%%%%%%%%%
\begin{align*}
p_j^x\sigma_x\left(p_j^{k_j}\right)<\frac{1}{t}\sigma_x(n) 
\end{align*}
%%%%%%%%%%%%%%%%%%%%%%%%%%%%%%%%%%%%%%%%%%%%%%
and
%%%%%%%%%%%%%%%%%%%%%%%%%%%%%%%%%%%%%%%%%%%%%%
\begin{align*}
\frac{1}{p_j^x\sigma_x\left(p_j^{k_j}\right)} > \frac{t}{\sigma_x(n)}. 
\end{align*}
%%%%%%%%%%%%%%%%%%%%%%%%%%%%%%%%%%%%%%%%%%%%%%
Since $d^x$ divides $\sigma_x\left(p_j^{k_j}\right)$, $d^x<p_j^x\sigma_x\left(p_j^{k_j}\right)$, this gives us $\frac{1}{d^x}>\frac{1}{p_j^x\sigma_x\left(p_j^{k_j}\right)}$. We can rewrite the inequality $I(x,m_0)>I(x,d)$ as
%%%%%%%%%%%%%%%%%%%%%%%%%%%%%%%%%%%%%%%%%%%%%%
\begin{align*}
I(x,m_0)&>1+\frac{1}{d^x} \\
&>1+ \frac{1}{p_j^x\sigma_x\left(p_j^{k_j}\right)}\\
&>1+\frac{t}{\sigma_x(n)} = \frac{\sigma_x(n)+t}{\sigma_x(n)} = \frac{\sigma_x(n)+t}{I\left(x,p_j^{k_j}\right)I\left(x,\frac{n}{p_j^{k_j}}\right)n^x}\\
&\geq \frac{\sigma_x(n)+t}{I\left(x,p_j^{k_j}\right)I\left(x,\frac{n}{p_j^{k_j}}\prod_{i=1}^{r} p_i^{q_i}\right)n^x}.
\end{align*}
%%%%%%%%%%%%%%%%%%%%%%%%%%%%%%%%%%%%%%%%%%%%%%
From our previous assumption, 
%%%%%%%%%%%%%%%%%%%%%%%%%%%%%%%%%%%%%%%%%%%%%%
\begin{align*}
I(x,a)= I\left(x,p_j^{k_j}\right)I\left(x,m_0\right)I\left(x,\frac{n}{p_j^{k_j}}\prod_{i=1}^{r} p_i^{q_i}\right).
\end{align*}
%%%%%%%%%%%%%%%%%%%%%%%%%%%%%%%%%%%%%%%%%%%%%%
Substituting our previous inequality, 
%%%%%%%%%%%%%%%%%%%%%%%%%%%%%%%%%%%%%%%%%%%%%%
\begin{align*}
I(x,a) > \frac{\sigma_x(n)+t}{n^x}.
\end{align*}
%%%%%%%%%%%%%%%%%%%%%%%%%%%%%%%%%%%%%%%%%%%%%%
Therefore, we have a contradiction and $\frac{\sigma_x(n)+t}{n^x}$ is an \xth abundancy outlaw.
%%%%%%%%%%%%%%%%%%%%%%%%%%%%%%%%%%%%%%%%%%%%%%
\end{proof}
%%%%%%%%%%%%%%%%%%%%%%%%%%%%%%%%%%%%%%%%%%%%%%
Theorem \ref{Theorem2} produces a class of \xth abundancy outlaws of the form $\frac{\sigma_x(n)+t}{n^x}$ where $t$ is a positive integer. Our next goal is to find \xth abundancy outlaws lying within a certain range. We note how much the \xth abundancy index of a positive integer $n=\prod_{i=1}^{s} p_i^{k_i}$ increases by multiplying $n$ by one of its prime power factors $p_i^{k_i}$, where $1\leq i \leq s$. We then present a theorem determining \xth abundancy outlaws $\frac{a}{b}$ falling within the range $I(x,n) < \frac{a}{b} < I(x,p_jn)$ where $n$ is a positive integer and $p_j$ a prime power factor of $n$.
%%%%%%%%%%%%%%%%%%%%%%%%%%%%%%%%%%%%%%%%%%%%%%

%%%%%%%%%%%%%%%%%%%%%%%%%%%%%%%%%%%%%%%%%%%%%%
\begin{lem} \label{Lemma1}
Let $n$ be a positive integer with $n=\prod_{i=1}^{s} p_i^{k_i}$ for primes $p_1, p_2,\ldots, p_s$. Then 
\begin{align*}
\frac{I\left(x,p_jn\right)}{I\left(x,n\right)}=\frac{\sigma_x\left(p_j^{k_j+1}\right)}{\sigma_x\left(p_j^{k_j+1}\right)-1}  
\end{align*}
for all  $1 \leq j \leq s$. 
\end{lem}
%%%%%%%%%%%%%%%%%%%%%%%%%%%%%%%%%%%%%%%%%%%%%%
\begin{proof}
%%%%%%%%%%%%%%%%%%%%%%%%%%%%%%%%%%%%%%%%%%%%%%
Let $n$ be a positive integer with $n=\prod_{i=1}^{s} p_i^{k_i}$ for primes $p_1, p_2,\ldots, p_s$. Then 
%%%%%%%%%%%%%%%%%%%%%%%%%%%%%%%%%%%%%%%%%%%%%%
\begin{align*}
\frac{I\left(x,p_jn\right)}{I\left(x,n\right)} &=\frac{\sigma_x(p_jn)}{p_j^x\sigma_x(n)} = \frac{\sigma_x\left(p_j^{k_j+1}\right)\sigma_x\left(\frac{n}{p_j^{k_j}}\right)}{p_j^x\sigma_x\left(p_j^{k_j}\right)\sigma_x\left(\frac{n}{p_j^{k_j}}\right)} = \frac{\sigma_x\left(p_j^{k_j+1}\right)}{p_j^x\sigma_x\left(p_j^{k_j}\right)} = \frac{\sigma_x\left(p_j^{k_j+1}\right)}{\sigma_x\left(p_j^{k_j+1}\right)-1}. 
\end{align*}
%%%%%%%%%%%%%%%%%%%%%%%%%%%%%%%%%%%%%%%%%%%%%%
Therefore, $\frac{I\left(x,p_jn\right)}{I\left(x,n\right)}=\frac{\sigma_x\left(p_j^{k_j+1}\right)}{\sigma_x\left(p_j^{k_j+1}\right)-1}$ for all $1 \leq j \leq s$.
%%%%%%%%%%%%%%%%%%%%%%%%%%%%%%%%%%%%%%%%%%%%%%
\end{proof}
%%%%%%%%%%%%%%%%%%%%%%%%%%%%%%%%%%%%%%%%%%%%%%

%%%%%%%%%%%%%%%%%%%%%%%%%%%%%%%%%%%%%%%%%%%%%%

%%%%%%%%%%%%%%%%%%%%%%%%%%%%%%%%%%%%%%%%%%%%%%
\e*
%%%%%%%%%%%%%%%%%%%%%%%%%%%%%%%%%%%%%%%%%%%%%%
\begin{proof}
%%%%%%%%%%%%%%%%%%%%%%%%%%%%%%%%%%%%%%%%%%%%%%
Let $\frac{k}{lm^x}$ be a fraction greater than one such that $(k,lm^x)=1$. Suppose there exists a divisor $n^x= \prod_{i=1}^{s} p_i^{xk_i}$ of $lm^x$ such that  
%%%%%%%%%%%%%%%%%%%%%%%%%%%%%%%%%%%%%%%%%%%%%%
\begin{enumerate}
\item $\frac{k}{lm^x}<I\left(x,p_in\right)$ for all $1\leq i\leq s$; and 
\item $\sigma_x(n)l\left(\frac{m}{n}\right)^x$  has a divisor $d^x$ such that $(d^x,k)=1$ and $I(x,d)\geq \frac{\sigma_x\left(p_j^{k_j+1}\right)}{\sigma_x\left(p_j^{k_j+1}\right)-1}$ for some positive integer $1\leq j\leq s$. 
\end{enumerate}
%%%%%%%%%%%%%%%%%%%%%%%%%%%%%%%%%%%%%%%%%%%%%%
For sake of contradiction, suppose $\frac{k}{lm^x}$ is an \xth abundancy index. This implies that $I\left(x,a\right) = \frac{k}{lm^x}$ for some integer $a$ and  
 %%%%%%%%%%%%%%%%%%%%%%%%%%%%%%%%%%%%%%%%%%%%%%
\begin{equation*}
lm^x\sigma_x(a) = ka^x. 
\end{equation*}
%%%%%%%%%%%%%%%%%%%%%%%%%%%%%%%%%%%%%%%%%%%%%%
From our initial assumption and Property \ref{Property1}, $n^x$ divides $lm^x$, which gives us that $n^x$ divides $a^x$. Using Property \ref{Property2}, $n$ divides $a$, hence, $a=bn$ for some integer $b$. We also have that $\frac{k}{lm^x}<I\left(x,p_in\right)$ for all $\ 1\leq i\leq s$, which implies 
%%%%%%%%%%%%%%%%%%%%%%%%%%%%%%%%%%%%%%%%%%%%%%
\begin{align*}
\frac{k}{lm^x} = I(x,a)<\frac{\sigma_x(p_in)}{(p_in)^x} 
\end{align*}
%%%%%%%%%%%%%%%%%%%%%%%%%%%%%%%%%%%%%%%%%%%%%%
and
%%%%%%%%%%%%%%%%%%%%%%%%%%%%%%%%%%%%%%%%%%%%%%
\begin{align*} 
(p_in)^x\sigma_x(a) < a^x\sigma_x(p_in), 
\end{align*}
%%%%%%%%%%%%%%%%%%%%%%%%%%%%%%%%%%%%%%%%%%%%%%
which gives us $p^{x(k_i+1)}$ does not divide $ a^x$ for all $1\leq i\leq s$. By Property \ref{Property2}, $p^{k_i+1}$ does not divide $a$ for all $1\leq i\leq s$, this implies $(b,n)=1$. Since $I(x,n)$ is multiplicative, 
%%%%%%%%%%%%%%%%%%%%%%%%%%%%%%%%%%%%%%%%%%%%%%
\begin{equation*}
I\left(x,a\right) = I\left(x,bn\right) = I\left(x,b\right)I\left(x,n\right)=\frac{k}{lm^x}.
\end{equation*}
%%%%%%%%%%%%%%%%%%%%%%%%%%%%%%%%%%%%%%%%%%%%%%
It follows that
%%%%%%%%%%%%%%%%%%%%%%%%%%%%%%%%%%%%%%%%%%%%%%
\begin{equation*}
\sigma_x(b)\sigma_x(n)l\left(\frac{m}{n}\right)^x=kb^x.
\end{equation*}
%%%%%%%%%%%%%%%%%%%%%%%%%%%%%%%%%%%%%%%%%%%%%%
We know that there exists a positive integer $d^x$ such that $d^x$ divides $\sigma_x(n)\left(\frac{m}{n}\right)^x$ and  $(d^x,k)=1.$ By Properties \ref{Property1} and \ref{Property2}, $d$ divides $b$. From Proposition \ref{Proposition1} and Lemma \ref{Lemma1},  
%%%%%%%%%%%%%%%%%%%%%%%%%%%%%%%%%%%%%%%%%%%%%%
\begin{align*}
I(x,b)I(x,n) > I(x,d)I(x,n),
\end{align*}
%%%%%%%%%%%%%%%%%%%%%%%%%%%%%%%%%%%%%%%%%%%%%%
implying
%%%%%%%%%%%%%%%%%%%%%%%%%%%%%%%%%%%%%%%%%%%%%%
\begin{align*}
I(x,d)\geq  \frac{\sigma_x\left(p_j^{k_j+1}\right)}{\sigma_x\left(p_j^{k_j+1}\right)-1} = \frac{\sigma_x(p_jn)}{p_j^x\sigma_x(n)}
\end{align*}
%%%%%%%%%%%%%%%%%%%%%%%%%%%%%%%%%%%%%%%%%%%%%%
for some positive integer $1\leq j\leq s$. Hence, 
%%%%%%%%%%%%%%%%%%%%%%%%%%%%%%%%%%%%%%%%%%%%%%
\begin{align*}
I(x,a)= \frac{k}{lm^x}> I(x,d)I(x,n)\geq \left(\frac{\sigma_x(p_jn)}{p_j^x\sigma_x(n)}\right)\left(\frac{\sigma_x(n)}{n^x}\right)=\frac{\sigma_x(p_jn)}{(p_jn)^x}=I(x,p_jn).
\end{align*}
%%%%%%%%%%%%%%%%%%%%%%%%%%%%%%%%%%%%%%%%%%%%%%
Therefore, we have a contradiction and $\frac{k}{lm^x}$ is an \xth abundancy outlaw. 
%%%%%%%%%%%%%%%%%%%%%%%%%%%%%%%%%%%%%%%%%%%%%%
\end{proof}
%%%%%%%%%%%%%%%%%%%%%%%%%%%%%%%%%%%%%%%%%%%%%%
Through these theorems, we have located certain rationals greater than one that fail to be in the image of the function $I(x,n)$. The next question we consider is when rationals greater than one are the abundancy index of at least one positive integer.
%%%%%%%%%%%%%%%%%%%%%%%%%%%%%%%%%%%%%%%%%%%%%%

%%%%%%%%%%%%%%%%%%%%%%%%%%%%%%%%%%%%%%%%%%%%%%
\section{\xth Abundancy Indices}
%%%%%%%%%%%%%%%%%%%%%%%%%%%%%%%%%%%%%%%%%%%%%%
In this section, we observe rationals greater than one that fall into the image of the function $I(x,n)$. Our first proposition looks at abundancy outlaws that are \xth abundancy indices.
%%%%%%%%%%%%%%%%%%%%%%%%%%%%%%%%%%%%%%%%%%%%%%
\begin{prop}
If $p$ is prime and  $x>1$, then $I(x,p)$ is an abundancy outlaw. 
\end{prop}
%%%%%%%%%%%%%%%%%%%%%%%%%%%%%%%%%%%%%%%%%%%%%%
In particular, $I(x,p)$ is an abundancy outlaw but is an \xth abundancy index when $p$ is prime. 
%%%%%%%%%%%%%%%%%%%%%%%%%%%%%%%%%%%%%%%%%%%%%%
\begin{proof}
%%%%%%%%%%%%%%%%%%%%%%%%%%%%%%%%%%%%%%%%%%%%%%
Let $p$ be prime and $x >1$, then
%%%%%%%%%%%%%%%%%%%%%%%%%%%%%%%%%%%%%%%%%%%%%%
\begin{align*}
I(x,p) &= \frac{\sigma_x(p)}{p^x} \\ 
&= \frac{1+p^x}{p^x}.
\end{align*} 
%%%%%%%%%%%%%%%%%%%%%%%%%%%%%%%%%%%%%%%%%%%%%%
We note that 
%%%%%%%%%%%%%%%%%%%%%%%%%%%%%%%%%%%%%%%%%%%%%%
\begin{align*}
p^x<1+p^x < \sigma(p^x)= \sum_{i=0}^{x} {p^{i}}.
\end{align*}
%%%%%%%%%%%%%%%%%%%%%%%%%%%%%%%%%%%%%%%%%%%%%%
From Theorem \ref{Theorem1}, $I(x,p)$ is an abundancy outlaw. Therefore, the abundancy outlaw $I(x,p)$ is in the image of $I(x,n)$ where $x> 1$.
%%%%%%%%%%%%%%%%%%%%%%%%%%%%%%%%%%%%%%%%%%%%%%
\end{proof}
%%%%%%%%%%%%%%%%%%%%%%%%%%%%%%%%%%%%%%%%%%%%%%
The next theorem and corollary are generalizations of Holdener's and Czarnecki's work \cite{Czarnecki}. They allow us to determine whether certain rationals greater than one are the \xth abundancy index of at least one positive integer.
%%%%%%%%%%%%%%%%%%%%%%%%%%%%%%%%%%%%%%%%%%%%%%
\c*
%%%%%%%%%%%%%%%%%%%%%%%%%%%%%%%%%%%%%%%%%%%%%%
\begin{proof}
Let $\frac{a}{cb^x}$ be a fraction greater than one in simplest terms. Suppose $\frac{a}{cb^x}=I(x,n)$ for some positive integer $n$ and $cb^x$ has a divisor $d^x=\prod_{i=1}^{s} p_i^{xk_i}$ such that $I(x,p_id)>\frac{a}{cb^x}$ for all $1 \leq i \leq s$. Suppose further that $I(x,n)=\frac{a}{cb^x}$ for some positive integer $n$, then
%%%%%%%%%%%%%%%%%%%%%%%%%%%%%%%%%%%%%%%%%%%%%%%%%
\begin{equation*}
cb^x\sigma_x(n) = an^x. 
\end{equation*}
%%%%%%%%%%%%%%%%%%%%%%%%%%%%%%%%%%%%%%%%%%%%%%
From Property \ref{Property1} and our initial assumption, $d^x$ divides $n^x$. Using Property \ref{Property2}, $d$ divides $n$, hence, $n=md$ for some integer $m$. We return to our initial assumption, $I(x,p_id)>\frac{a}{cb^x}$ for all $1 \leq i \leq s$, which implies
%%%%%%%%%%%%%%%%%%%%%%%%%%%%%%%%%%%%%%%%%%%%%%
\begin{align*}
\frac{a}{cb^x} = I(x,n)<\frac{\sigma_x(p_in)}{(p_in)^x} 
\end{align*}
and 
\begin{equation*}
(p_in)^x\sigma_x(n) <n^x\sigma_x(p_in).
\end{equation*}
%%%%%%%%%%%%%%%%%%%%%%%%%%%%%%%%%%%%%%%%%%%%%%
This gives us $p^{x(k_i+1)}$ does not divide $ n^x$ for all $1\leq i\leq s$. Using Property \ref{Property2}, $p^{k_i+1}$ does not divide $n$ for all $1\leq i\leq s$, it follows that $(m,d)=1$. Since $I(x,n)$ is multiplicative, 
%%%%%%%%%%%%%%%%%%%%%%%%%%%%%%%%%%%%%%%%%%%%%%
\begin{align*}
& I\left(x,n\right) = I\left(x,md\right) = I\left(x,m\right)I\left(x,d\right)=\frac{a}{cb^x}.
\end{align*}
This implies 
\begin{equation*}
\frac{\sigma_x(m)}{m^x}\frac{\sigma_x(d)}{d^x}=\frac{a}{cb^x} 
\end{equation*}
and 
\begin{equation*}
\frac{\sigma_x(m)}{m^x}=\frac{d^x}{\sigma_x(d)}\frac{a}{cb^x}, 
\end{equation*}
%%%%%%%%%%%%%%%%%%%%%%%%%%%%%%%%%%%%%%%%%%%%%%
giving us $I(x,m) = \frac{d^x}{\sigma_x(d)}\frac{a}{cb^x}$. Therefore, $\frac{d^x}{\sigma_x(d)}\frac{a}{cb^x}$ is an \xth abundancy index.
%%%%%%%%%%%%%%%%%%%%%%%%%%%%%%%%%%%%%%%%%%%%%%
\end{proof}
%%%%%%%%%%%%%%%%%%%%%%%%%%%%%%%%%%%%%%%%%%%%%%
\begin{cor}
Let $m, n, t$ be positive integers. If $\frac{\sigma_x(mn)+\sigma_x(m)t}{(mn)^x}$ is a fraction in simplest terms with $m^x= \prod_{i=1}^{s} p_i^{xk_i}$ and $I(x,p_im)>\frac{\sigma_x(mn)+\sigma_x(m)t}{(mn)^x}$ for all $1\leq i \leq s$, then $\frac{\sigma_x(n)+t}{n^x}$  is an \xth abundancy index if $\frac{\sigma_x(mn)+\sigma_x(m)t}{(mn)^x}$ is an \xth abundancy index. 
\end{cor}
%%%%%%%%%%%%%%%%%%%%%%%%%%%%%%%%%%%%%%%%%%%%%%
\begin{proof}
%%%%%%%%%%%%%%%%%%%%%%%%%%%%%%%%%%%%%%%%%%%%%%
Let $m, n, t$ be positive integers. Suppose $\frac{\sigma_x(mn)+\sigma_x(m)t}{(mn)^x}$ is a fraction in simplest terms with $m^x= \prod_{i=1}^{s} p_i^{xk_i}$ and $I(x,p_im)>\frac{\sigma_x(mn)+\sigma_x(m)t}{(mn)^x}$ for all $1\leq i \leq s$. Suppose further that $I(x,a)=\frac{\sigma_x(mn)+\sigma_x(m)t}{(mn)^x}$ for some positive integer $a$, then
%%%%%%%%%%%%%%%%%%%%%%%%%%%%%%%%%%%%%%%%%%%%%%
\begin{equation*}
(mn)^x\sigma_x(a) = a^x\sigma_x(mn)+\sigma_x(m)t.
\end{equation*}
%%%%%%%%%%%%%%%%%%%%%%%%%%%%%%%%%%%%%%%%%%%%%%
Using Properties \ref{Property1} and \ref{Property2}, $mn$ divides $a$, which gives us $a=bmn$ for some integer $b$. From our initial assumption, $m^x= \prod_{i=1}^{s} p_i^{xk_i}$ and $I(x,p_im)>\frac{\sigma_x(mn)+\sigma_x(m)t}{(mn)^x}$ for all $1\leq i \leq s$, it follows that $(m,n)=1$. Hence
%%%%%%%%%%%%%%%%%%%%%%%%%%%%%%%%%%%%%%%%%%%%%%
\begin{align*}
I\left(x,a\right) = I\left(x,bmn\right) = I\left(x,m\right)I\left(x,bn\right)=\frac{\sigma_x(mn)+\sigma_x(m)t}{(mn)^x}.
\end{align*}
%%%%%%%%%%%%%%%%%%%%%%%%%%%%%%%%%%%%%%%%%%%%%%
Since $\sigma_x$ is multiplicative, can rewrite the equation as
%%%%%%%%%%%%%%%%%%%%%%%%%%%%%%%%%%%%%%%%%%%%%%
\begin{align*}
\frac{\sigma_x(mn)+\sigma_x(m)t}{(mn)^x} &= \frac{\sigma_x(m)\sigma_x(n)+\sigma_x(m)t}{(mn)^x} \\
&= \frac{\sigma_x(m)(\sigma_x(n)+t)}{(mn)^x} \\ 
&= I(x,m)\frac{\sigma_x(n)+t}{n^x}.
\end{align*}
%%%%%%%%%%%%%%%%%%%%%%%%%%%%%%%%%%%%%%%%%%%%%%
We now have
%%%%%%%%%%%%%%%%%%%%%%%%%%%%%%%%%%%%%%%%%%%%%%
\begin{align*}
I\left(x,m\right)I\left(x,bn\right) = I\left(x,m\right)\frac{\sigma_x(n)+t}{n^x} 
\end{align*}
%%%%%%%%%%%%%%%%%%%%%%%%%%%%%%%%%%%%%%%%%%%%%%
which implies
%%%%%%%%%%%%%%%%%%%%%%%%%%%%%%%%%%%%%%%%%%%%%%
\begin{equation*}
I\left(x,bn\right) = \frac{\sigma_x(n)+t}{n^x}.
\end{equation*}
%%%%%%%%%%%%%%%%%%%%%%%%%%%%%%%%%%%%%%%%%%%%%%
Therefore, $\frac{\sigma_x(n)+t}{n^x}$ is an \xth abundancy index. 
\end{proof}
%%%%%%%%%%%%%%%%%%%%%%%%%%%%%%%%%%%%%%%%%%%%%%

%%%%%%%%%%%%%%%%%%%%%%%%%%%%%%%%%%%%%%%%%%%
\section{Acknowledgements}
%%%%%%%%%%%%%%%%%%%%%%%%%%%%%%%%%%%%%%%%%%%
The author would like to dedicate this paper to his friend and first research mentor at Emory University, Dr.~Paul Bruno, Case Western Reserve University ('16). In addition, he would also like to thank Dr.~Mark Norfleet, Jelisa Tan, and the anonymous reviewer for research and writing advice. Finally, he would like to thank the Emory Department of Mathematics and Dr.~David Zureick-Brown for sponsoring his research. This paper was inspired by \cite{Holdener}.
%%%%%%%%%%%%%%%%%%%%%%%%%%%%%%%%%%%%%%%%%%%

%%%%%%%%%%%%%%%%%%%%%%%%%%%%%%%%%%%%%%%%%%%
\bibliography{Generalizing}{}
\bibliographystyle{plain}
%%%%%%%%%%%%%%%%%%%%%%%%%%%%%%%%%%%%%%%%%%%

\end{document}